\documentclass[english]{smfart}
\usepackage{smfthm-treplo}
\usepackage{amssymb}
\usepackage{graphicx}
\usepackage[matrix, arrow,all,cmtip,color]{xy}
\usepackage{color}
\usepackage{fancyvrb}
\usepackage{MnSymbol}
\usepackage{hyperref}

\def\rtt{\operatorname\rightthreetimes}
\def\rrt#1#2#3#4#5#6#7{\xymatrix{ {#1} \ar[r]^{} \ar@{->}[d]_{#2} & {#4} \ar[d]^{#5} \\ {#3}  \ar[r] \ar@{-->}[ur]^{#7}& {#6} }}

\def\lra{\longrightarrow}
\def\lr{{\rtt lr}}
\def\lrl{{\rtt l}}
\def\rl{{\rtt rl}}
\def\rlr{{\rtt r}}

\def\barv{\overline{v}}\def\barw{\overline{w}}

\def\lra{\longrightarrow}

\def\ra{\searrow}
\def\la{\swarrow}

\def\llrra{\leftrightarrow}

\def\xra{\xrightarrow}

\title[extremally disconnected as 
$\left\{\{u\rotatebox{-12}{\ensuremath{\to}}\raisebox{-2pt}{\ensuremath{a\,\,,\hskip-6pt\raisebox{3pt}{\bf\color{red}=}b}\rotatebox{13}{\ensuremath{\leftarrow}}} u\}\right\}^l$,
and proper as $\left(\left\{\{o\rotatebox{-12}{\bf\color{red}\ensuremath{\mathbf{\to}}}\raisebox{-2pt}{\color{red}\ensuremath{\mathbf{c}}}\}\right\}^r_{<4}\right)^{lr}$]
{Extremally disconnected spaces as $\{\{u\rotatebox{-12}{\ensuremath{\to}}\raisebox{-2pt}{\ensuremath{a,b}\rotatebox{13}{\ensuremath{\leftarrow}}} u\}\lra \{u\rotatebox{-12}{\ensuremath{\to}}\raisebox{-2pt}{\ensuremath{a=b}\rotatebox{12}{\ensuremath{\leftarrow}}}v\}\}^l$,
and being proper as $(\{\{o\}\lra\{o\to c\}\}^r_{<4})^{lr}$}
\author{masha gavrilovich}
\thanks{\tiny Die Mathematiker sind eine Art Franzosen: Redet man zu ihnen, so \"ubersetzen sie es in ihre Sprache, und dann ist es alsobald ganz etwas anderes.---Johann Wolfgang von Goethe. Maximen und Reflexionen. 
Aphorismen und Aufzeichnungen. Nach den Handschriften des Goethe- und Schiller-Archivs hg. von Max Hecker, Verlag der Goethe-Gesellschaft, 
Weimar 1907, 
Aus dem Nachlass, Nr.~1005, Uber Natur und Naturwissenschaft.}
\address{Institute for Regional Economics Studies//IRESRAS}
\email{mishap@s\!\!\!sdf.org}
\urladdr{http://mishap.sdf.org/yetanothernotanobfuscatedsyntax.pdf}
\begin{document}
\begin{abstract}
We observe that the notions of a topological space being  extremally disconnected,
and of a continuous map of compact Hausdorff spaces being proper, 
can each be defined in terms of the Quillen lifting property  with respect to 
a surjective proper morphism of finite topological spaces, i.e.~in terms of a monotone map of finite preorders. 
This reveals the preorders
implicit in the statement of the Gleason theorem that 
extremally disconnected spaces are projective in the category of 
compact Hausdorff topological spaces, and interprets it 
as an instance of a weak factorisation system generated by an explicitly given morphism.
\end{abstract}
\maketitle
\enlargethispage{6\baselineskip}
\section{Introduction}
We observe that the notions of a topological space being  extremally disconnected, having closed points,  
and a continuous map of compact Hausdorff spaces being proper,\newline\begin{minipage}{0.7\textwidth} 
and being surjective proper, can each be defined in terms of the Quillen lifting property  with respect to 
a surjective proper morphism of finite topological spaces, i.e.~in terms of a monotone map of finite preorders (see Fig.~1). 
Based on this, we introduce a concise, 
and, in a sense intuitive, combinatorial notation expressing these notions via simplest (counter)examples,
and often closely following the standard definitions. We hope our results suggest that
this notation can be used to formalise these properties. 
\end{minipage}\begin{minipage}{0.3\textwidth}
\includegraphics[width=0.9\textwidth,angle=0,origin=c]{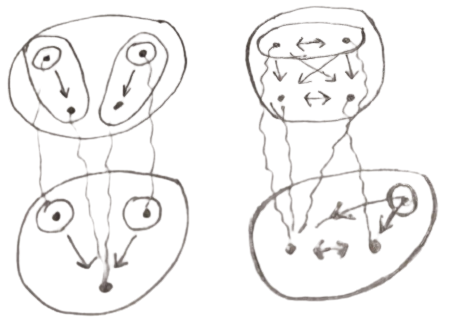}
\tiny	\\.  \ \ \ \ \ \ \ \ (a) \ \ \ \ \ \ \ \ \ \ \ \ \ \ (b)
        \\.  \  Fig.1. Maps of preorders for
	\\.\ \ \ \ (a) extr.disconnected\\.\ \ \ \ (b) proper
\end{minipage}

This allows us to write a couple of facts in general topology 
mentioned in the lecture notes on condensed mathematics by P.Scholze
in a category-theoretic language hopefully closer to the spirit of the notes. 
The theorem of [Gleason], cf.~[Condensed,Def.2.4], 
that extremally disconnected spaces are projective in the category of compact Hausdorff spaces,
can then be seen as saying that there is a weak factorisation system generated by a certain 
surjective proper morphism of finite topological spaces. Our reformulation allows us 
to see that it is important for this theorem that the map of preorders 
implicit in the definition of extremally disconnected, is both surjective and proper.
The fact that a surjective proper map is necessarily a quotient in our notation is expressed as 
an inclusion of orthogonals of morphisms of finite spaces: 
$$\hskip-42pt
\left\{\{u\rotatebox{-12}{\ensuremath{\to}}\raisebox{-2pt}{\ensuremath{a,b}\rotatebox{13}{\ensuremath{\leftarrow}}} v\}\lra \{u\rotatebox{-12}{\ensuremath{\to}}\raisebox{-2pt}{\ensuremath{a=b}\rotatebox{12}{\ensuremath{\leftarrow}}}v\},
\{a\llrra b\rotatebox{-12}{\ensuremath{\to}}\raisebox{-2pt}{\ensuremath{c\llrra d}}\}\lra 
\{a\llrra b=c=d\}\right\}^\lr \subset 
\left\{\{o\rotatebox{-12}{\ensuremath{\to}}\raisebox{-2pt}{\ensuremath{c}}\}\lra \{o\llrra c\}\right\}^\lrl$$
This fact is "the key point" [Analytic, p.7] in the sheaf condition holding for condensed sets represented by topological spaces.  

Lemma~\ref{lem:lrlp} and Lemma~\ref{lem:surjpr} state our reformulations; \S2 introduces necessary notation
and reformulations in terms of lifting properties.
Appendix A gives a list of reformulations of elementary notions in general topology in terms of lifting properties. 
In an unfinished appendix B we attempt to present a diagram chasing rendering of the proof of 
Lemma 1.3 [Analytic] stating that product commutes with filtered colimits in the category of compact Hausdorff spaces
whenever all maps are (closed) inclusions, to help the reader ponder whether our notation can be used in an efficient
theorem prover for elementary topology.

\subsection{Further questions} Our reformulations reveal combinatorics 
of finite preorders implicit, perhaps surprisingly, 
in many standard definitions in elementary general topology. 
It may be interesting to understand this combinatorics or make use of it.

How much of elementary topology could be developed 
entirely combinatorially? Say, could one prove the Gleason theorem 
entirely by a combinatorial diagram chasing calculation ?  
Or the combinatorial expression above representing the fact that 
surjective proper maps are necessarily quotient.

Could these combinatorial expressions, say for quotients or properness, be interpreted 
in larger categories of spaces, say of condensed sets or the category of simplicial objects
in the category of filters [situs]? 

Formalisation of condensed mathematics may perhaps need the notions of 
extremally disjoint spaces or proper maps. Could our reformulation suggest
an efficient theorem prover/tactic for elementary claims about these notions?   
Note that notions defined by lifting properties behave nicely with respect 
to limits and colimits.

\subsection{Explaining notation in the title} The notation and necessary definitions are 
introduced in the next section. Here we give a brief sketch assuming familiarity with lifting properties. 

\subsubsection{Extremally disconnected.}
For a class of morphisms $P$, let $P^\rlr$ and $P^\lrl$ denote 
the class of morphisms having the left, resp.~right, lifting property with respect to each morphism in $P$.  
Let $P^\rl:=(P^\rlr)^\lrl$, $P^\lr=(P^\lrl)^\rlr$. Recall a topology on a finite set may be regarded as a preorder
or, equivalently, as a category with unique morphisms:
$x\leqslant y$, resp.~$x\searrow y$, iff $y$ lies in the closure of $x$.

Thus $\{u\rotatebox{-12}{\ensuremath{\to}}\raisebox{-2pt}{\ensuremath{a,b}\rotatebox{13}{\ensuremath{\leftarrow}}} v\}$
denotes the topological space with two open points $u$ and $v$, two closed points $a$ and $b$, 
split into two connected components 
$\{u\rotatebox{-12}{\ensuremath{\to}}\raisebox{-2pt}{\ensuremath{a}}\}$ 
and $\{b\rotatebox{13}{\ensuremath{\leftarrow}} v\}$.
The expression 
$\{u\rotatebox{-12}{\ensuremath{\to}}\raisebox{-2pt}{\ensuremath{a,b}\rotatebox{13}{\ensuremath{\leftarrow}}} v\}\lra \{u\rotatebox{-12}{\ensuremath{\to}}\raisebox{-2pt}{\ensuremath{a=b}\rotatebox{12}{\ensuremath{\leftarrow}}}v\}$ denotes the morphism 
gluing together points $a$ and $b$. In terms of categories, we think that this morphism is the functor 
``adding an identity morphism between objects $a$ and $b$''. Hence, 
$\{\{u\rotatebox{-12}{\ensuremath{\to}}\raisebox{-2pt}{\ensuremath{a,b}\rotatebox{13}{\ensuremath{\leftarrow}}} v\}\lra \{u\rotatebox{-12}{\ensuremath{\to}}\raisebox{-2pt}{\ensuremath{a=b}\rotatebox{12}{\ensuremath{\leftarrow}}}v\}\}^\lrl$
denotes the class of morphisms having the left lifting property with respect to $\{u\rotatebox{-12}{\ensuremath{\to}}\raisebox{-2pt}{\ensuremath{a,b}\rotatebox{13}{\ensuremath{\leftarrow}}} v\}\lra \{u\rotatebox{-12}{\ensuremath{\to}}\raisebox{-2pt}{\ensuremath{a=b}\rotatebox{12}{\ensuremath{\leftarrow}}}v\}$. To give a map to $\{u\rotatebox{-12}{\ensuremath{\to}}\raisebox{-2pt}{\ensuremath{a=b}\rotatebox{12}{\ensuremath{\leftarrow}}}v\}$ is the same as to give two disjoint open subsets of $E$ (the preimages of $u$ and $v$). 
It lifts to $\{u\rotatebox{-12}{\ensuremath{\to}}\raisebox{-2pt}{\ensuremath{a,b}\rotatebox{13}{\ensuremath{\leftarrow}}} v\}$ 
iff $U$ and $V$ lie in disjoint closed and open subsets. This is one of the equivalent definitions of $E$ being extremally disconnected. 

Note that the morphism 
$\{u\rotatebox{-12}{\ensuremath{\to}}\raisebox{-2pt}{\ensuremath{a,b}\rotatebox{13}{\ensuremath{\leftarrow}}} v\}\lra \{u\rotatebox{-12}{\ensuremath{\to}}\raisebox{-2pt}{\ensuremath{a=b}\rotatebox{12}{\ensuremath{\leftarrow}}}v\}$ is surjective, closed
and thereby proper, being a morphism of finite spaces. Also note that being surjective and proper
are right lifting properties, and that this means that each map in 
$\{\{u\rotatebox{-12}{\ensuremath{\to}}\raisebox{-2pt}{\ensuremath{a,b}\rotatebox{13}{\ensuremath{\leftarrow}}} v\}\lra \{u\rotatebox{-12}{\ensuremath{\to}}\raisebox{-2pt}{\ensuremath{a=b}\rotatebox{12}{\ensuremath{\leftarrow}}}v\}\}^\lr$
is surjective and proper. Hence, the existence of the weak factorisation system generated by this morphism
implies that each map $\emptyset \lra E$ decomposes as $\emptyset \lra X \in 
\left\{\{u\rotatebox{-12}{\ensuremath{\to}}\raisebox{-2pt}{\ensuremath{a,b}\rotatebox{13}{\ensuremath{\leftarrow}}} v\}\lra \{u\rotatebox{-12}{\ensuremath{\to}}\raisebox{-2pt}{\ensuremath{a=b}\rotatebox{12}{\ensuremath{\leftarrow}}}v\}\right\}^\lrl$
and a surjective proper map $E\lra X \in 
\left\{\{u\rotatebox{-12}{\ensuremath{\to}}\raisebox{-2pt}{\ensuremath{a,b}\rotatebox{13}{\ensuremath{\leftarrow}}} v\}\lra \{u\rotatebox{-12}{\ensuremath{\to}}\raisebox{-2pt}{\ensuremath{a=b}\rotatebox{12}{\ensuremath{\leftarrow}}}v\}\right\}^\lr$. 
That is, each space admits a surjection from an extremally disconnected space,
and a compact space admits a surjection from a compact extremally disconnected space.
A similar argument shows that the extremally disconnected space can also be assumed Hausdorff,
based on a reformulation of Axiom T1 (having closed points) as 
 by a right lifting property, namely with respect to the morphism $\{o\ra c\}\lra \{o=c\}$
gluing the Sierpinski space into a single point.

\subsubsection{Proper.} The expression 
$\{o\}\lra\{o\to c\}$ denotes perhaps the simplest example of a non-proper (actually, not closed) map: 
the map sending a point into the open point of the Sierpinski space.
It is easy to check that a map of finite spaces is in $\{\{o\}\lra\{o\to c\}\}^r$, i.e.~has 
the right lifting property with respect to this map, iff it is closed.
By $\{\{o\}\lra\{o\to c\}\}^r_{<4}$ we denote the subclass of $\{\{o\}\lra\{o\to c\}\}^r$ consisting of maps of spaces with 
less than 4 points.
By  [Bourbaki, General Topology, I\S10.2,Th.1(d)] (see Lemma~\ref{lem:surjpr}(1)) being proper is a left lifting property,
thus $lr$-orthogonal of any class of proper morphisms is a class of proper morphisms. 
Hence, each morphism in $(\{\{o\}\lra\{o\to c\}\}^r_{<4})^{lr}$ is proper, and by Engelking-Taimanov theorem 
(see Lemma~\ref{lem:surjpr}(3))) any morphism of compact Hausdorff spaces is in this class.

\section{Extremally disconnected sets being projective as a weak factorisation system}

\subsection{Preliminary lemmas} A number of definitions in general topology can be expressed by applying several times
the Quillen lifting property to a morphism of finite topological spaces, thereby leading to a concise notation 
based on finite preorders and their maps; Appendix~A gives a list of such reformulations. 
In this subsection we introduce notations and state facts 
necessarily to express in this way the definitions of  extremally disconnected and being proper.

\subsubsection{Quillen lifting property}
Recall that a morphism $i$ in a category has the {\em left lifting property} with respect to a morphism $p$, 
and $p$ also has the {\em right lifting property} with respect to $i$, 
iff for each $f:A\to X$ and $g:B\to Y$ such that $p\circ f = g \circ i$ there exists $h:B\to X$ such that $h\circ i = f$ and $p\circ h = g$.


For a class $P$ of morphisms in a category, its {\em left orthogonal} $P^{\rtt l}$ with respect to the lifting property, respectively its {\em right orthogonal} $P^{\rtt r}$, is the class of all morphisms which have the left, respectively right, lifting property with respect to each morphism in the class $P$. In notation,
$$
P^{\rtt l} := \{ i \,\,:\,\, \forall p\in P\,\, i\rtt p\},
P^{\rtt r} := \{ p \,\,:\,\, \forall i\in P\,\, i\rtt p\}, P^\lr:=(P^\lrl)^\rlr,..$$

Taking the orthogonal of a class $P$ is a simple way to define a class of morphisms excluding non-isomorphisms 
from $P$, in a way which is useful in a diagram chasing computation, and is often used to
define properties of morphisms starting from an explicitly given class of (counter)examples. 
  For this reason, 
it is convenient to refer to $P^\lrl$ and $P^\rlr$ as {\em left, resp.~right, Quillen negation} of property $P$. 

\subsubsection{Notation for finite topological spaces and their morphisms}

A topological space comes with a {\em specialisation preorder} on its points: for
points $x,y \in X$,  $x \leq y$ iff $y \in cl x$ ($y$ is in the closure of $x$), 
or equivalently. The resulting preordered set may be regarded as a category whose
objects are the points of ${X}$ and where there is a unique morphism $x{\searrow}y$ iff $y \in cl x$.

For a finite topological space $X$, the specialisation preorder or equivalently the corresponding category uniquely determines the space: a subset of ${X}$ is closed iff it is
downward closed, or equivalently, there are no morphisms going outside the subset.

The monotone maps (i.e.~functors) are the continuous maps for this topology.

We denote a finite topological space by a list of the arrows (morphisms) in
the corresponding category; arrows $\rightarrow$ and $\ra$ are interchangable 
and denote a morphism; '$\leftrightarrow $' denotes an isomorphism and '$=$' denotes the identity morphism.  An arrow between two such lists denotes a continuous map (a functor) which sends each point to the correspondingly labelled point, but possibly turning some morphisms into identity morphisms, thus gluing some points. This notation leads to a formal syntax defining morphisms of finite topological space, and, to emphasize this, we sometimes typeset these expressions as code: \verb|{a<->b}| for $\{a\llrra b\}$, and 
\verb|{a->b}| for $\{a\ra b\}$ etc. 

With this notation, we may display continuous functions for instance between the discrete space on two points, the Sierpinski space, the antidiscrete space and the point space as follows (where each point is understood to be mapped to the point of the same name in the next space):
$$  \begin{array}{ccccccc}
  \{a,b\}
     &\longrightarrow& 
  \{a{\searrow}b\}
     &\longrightarrow& 
  \{a\leftrightarrow b\}
    &\longrightarrow& 
  \{a=b\}
  \\ \verb|{a,b}|
     &\verb|-->|& 
  \verb|{a->b}|
     &\verb|-->|& 
  \verb|{a<->b}|
    &\verb|-->|& 
  \verb|{a=b}|
  \\
  \text{(discrete space)}
     &\longrightarrow& 
  \text{(Sierpinski space)}
    &\longrightarrow& 
  \text{(antidiscrete space)}
    &\longrightarrow& 
  \text{(single point)}
    \end{array}
$$
In $\{a{\searrow}b\}$, the point $a$ is open and point ${b}$ is closed.
\begin{rema} 
In $A \longrightarrow  B$, each object and each morphism in $A$ necessarily appears in $B$ as well.
We may extend the notation to avoid listing 
the same object or morphism twice, to make it more concise and easier to read, although 
at a cost of getting used to. For example, we may wish to shorten \verb|{u->a,b<-v}-->{u->a=b<-v}| to \verb|{u->a,b<-v}-->{a=b}|
or even $\{u \ra a\,\,,\hskip-6pt\raisebox{3pt}{\color{red}=}\, b\la v\}$ using red and placing symbols above each other 
to indicate morphisms and objects added.
Or perhaps to write 
$
\{a\} \longrightarrow  \{a,b\}
$ denoting the map from a single point to the discrete space with two points,
 as $\{a\} \longrightarrow  \{b\}
$ or $\{a{\color{red},b}\}$.  

Tricks like this can be useful in an actual implementation of this notation in a theorem prover.
\end{rema}

\subsubsection{Preliminary results}

We will use the following reformulations of properties of spaces and continuous maps in terms of lifting properties
with respect to morphisms of finite topological spaces. 
 
 Note that each notion is defined with help of a counterexample, often the simplest or archetypal one. 
A concise convenient way to express e.g.~items 3-5 is to say that 
{\em surjectivity, quotient, and injectivity are right Quillen negations of 
\verb${a}-->{a<->b}$, \verb${o->c}-->{o<->c}$, and \verb${a<->b}-->{a=b}$, resp.}

\begin{lemm}In the category of (all) topological spaces, the following holds.\begin{enumerate}
\item A map $X\longrightarrow Y$ is surjective iff \verb${}-->{o} /_ X-->Y$
\item Points are closed within each fibre of a map $X\lra Y$, i.e.~the fibres (as subspaces with induced topology)
satisfy separation axiom $T1$, 
iff $\verb|{o->c}-->{o=c} /_ X-->Y|$
\item A map $X\longrightarrow Y$ is surjective iff \verb$X-->Y /_ {a}-->{a<->b}$
\item A map $X\longrightarrow Y$ is a quotient iff \verb$X-->Y /_ {o->c}-->{o<->c}$
\item A map $X\longrightarrow Y$ is injective iff \verb$X-->Y /_ {a<->b}-->{a=b}$
\item A topological space $X$ is extremally disconnected iff 
\verb${}--> X /_ {u->a,b<-v}-->{u->a=b<-v}$
\item 
The topology on $X$ is induced from $Y$ along the map $X\lra Y$ 
iff  \verb$X-->Y /_  {o->c}-->{o=c}$
\item the map $X\lra Y$ has dense image iff 
\verb$X-->Y /_ {c}-->{o->c}$ 
\item 
The topology on $X$ is induced from $Y$ along the map $X\lra Y$ 
and this map is open 
iff  
\verb$X-->Y /_ {a<->b<-c}-->{a<->b=c}$
\item 
The topology on $X$ is induced from $Y$ along the map $X\lra Y$ 
and this map is closed 
iff  
\verb$X-->Y /_ {a<->b->c}-->{a<->b=c}$
%
\end{enumerate}
\label{lem:lrlp}\end{lemm}
\begin{proof} Verification is a matter of expanding the definitions. We do only a few of the items. 
1. This lifting property says that each point of $Y$ (the image of \verb$o$ in \verb$\{o}-->Y$) has a preimage.
2. Let $x,y\in X$ be arbitrary. 
The map sending \verb|o| to $x$ and \verb|c| to y, is continuous iff the closure of $x$ contains $y$. 
The commutativity of the square means $x$ and $y$ lie in the same fibre. The commutativity of the triangles means $x=y$.
3. This lifting property says that each decomposition of $X=A\cup B$ (the preimages 
of points \verb$a$ and \verb$b$ in \verb$X-->{a<->b}$) induces a decomposition $Y=p(A)\cup p(B)$ of $Y$ where we denote $p:X\longrightarrow Y$. This is evidently equivalent to injectivity. 
5. Recall that a topological space $X$ is extremally disconnected iff the closure of an open subset is closed.
This lifting property says for every two disjoint open subsets $U$ and $V$ of $X$ (the preimages of \verb$u$ and \verb$v$ in \verb$X-->{u->a=b<-v}$) there is a decomposition of $X$ into two closed and open subsets $U'$ and $V'$ (the preimages of subspaces 
\verb${u->a}$ and \verb$b<-v$ in \verb$X-->{u->a,b<-v}$) such that $U\subset U'$ and $V\subset V'$. If $X$ is extremally disconnected, 
then taking $U'$ to be the closure of $U$ gives the decomposition.

Alternatively but equivalently, the lifting property says that each $U\subset A\subset X$ where $U$ is open and $A$ is closed (the preimages of subspaces 
\verb$u$ and \verb$u->a$ in \verb$X-->{u->a,b<-v}$) is separated by 
a closed and open subset $A'$ such that $U\subset A' \subset A$. 
Take $A:=\bar U$ to be the closure of $U$; then necessarily $A'=A$.
Hence, the lifting property implies that the the closure of an open subset is open, 
i.e.~$X$ is extremally disconnected. 
The rest are analogous. 
\end{proof}

Following [Bourbaki, I\S6.5, Definition 5, Example], given an ultrafilter $\mathfrak U$ on the set of points of a space $B$, 
define $B\sqcup_{\mathcal U} \{\infty\}$ to be the space $B$ adjoined with a new closed point $\infty$
such that a subset is open iff it is either an open subset of $B$, or a union of $\{\infty\}$ 
and a $\mathfrak U$-big open subset of $B$. 

\begin{lemm}In the category of (all) topological spaces, the following holds.\begin{enumerate}
\item A map $X\lra Y$ is proper iff 
for each set $A$ viewed as a discrete topological space, each ultrafilter $\mathcal U$ on $A$ it holds 
$$B \longrightarrow B\sqcup_{\mathcal U} \{\infty\} \rtt X\to Y$$

\item The class $\{\verb|{u->a,b<-v}-->{u->a=b<-v}|\}^\lr$ is contained in the class of surjective proper morphisms, and, moreover,
\begin{itemize}
\item $\emptyset\lra E\in \{\verb|{u->a,b<-v}-->{u->a=b<-v}|\}^\lrl$ iff $E$ is extremally disconnected
\item if $X\lra Y \in \{\verb|{u->a,b<-v}-->{u->a=b<-v}|\}^\lr$ and $X$ is compact Hausdorff,  so is $Y$
\end{itemize}

\item Let $P$ be the set of all proper morphisms of finite topological spaces mentioned 
in the right lifting properties of Lemma~\ref{lem:lrlp} (i.e.~items 5-9), 
with or without item 6 (extremally disjoint).
The class $(P)^\lr$ consists of proper morphisms, 
and contains each proper morphism of compact Hausdorff spaces. 

\item Let $P'$ be the set of surjective proper morphisms of finite topological spaces mentioned 
in the right lifting properties of  Lemma~\ref{lem:lrlp} (i.e.~items 5-8). 
The class $(P')^\lr$ consists of surjective proper morphisms, 
and contains each surjective proper morphism of compact Hausdorff spaces.
\end{enumerate}
\label{lem:surjpr}\end{lemm}
\begin{proof}
1. 
[Bourbaki, General Topology, I\S10.2,Th.1(d)] almost states this lifting property: 
they take $A=|X|$ to be the the set of points of $X$ with discrete topology,
and the horizontal map $A\lra X$ in the square to be identity on points. 
An elementary argument shows that only the image in $X$ would matter,
and thus shows the equivalence of the statement by Bourbaki and this lifting property.
See [mintsGE, \S2.2.2] for details.  

2. The map \verb|{u->a,b<-v}-->{u->a=b<-v}| is both surjective and closed, 
which is the same as proper for maps of finite topological spaces. 
By Lemma~\ref{lem:lrlp} both being surjective and being proper 
are right Quillen negations.  
Hence, each map in $(P_0)^\lr$ is both surjective and proper, 
where $P_0:=\{\verb|{u->a,b<-v}-->{u->a=b<-v}|\}$ is the class
consisting of a single morphism \verb|{u->a,b<-v}-->{u->a=b<-v}|.
Lemma~\ref{lem:lrlp}(6) states that the lifting property defines
extremally disconnectedness. 
Moreover, note 
$\verb|{o->c}-->{o=c} /_ {u->a,b<-v}-->{u->a=b<-v}|$.
By Lemma~\ref{lem:lrlp}(2) this implies that if 
each point $Y$ is closed (i.e.~$Y$ satisfies separation axiom $T1$)
and $X\lra Y$ is in $(P_0)^\lr$, then each point of $X$ is closed. 
For compact spaces, axiom $T1$ implies $T2$ (being Hausdorff). Hence,
if $X\lra Y$ is in $(P_0)^\lr$ and 
if $Y$ is compact Hausdorff, so is $X$.

3.  Each morphism in $P$ is proper, hence the morphisms mentioned in item 1 are in $P^\lrl$,
  hence, again by item 1, each morphism in $P^\lr$ is proper. 
Lemma~\ref{lem:lrlp}(5-9) imply that $P^\lrl$ consists of inclusions $A\lra B$ where $A$ is a dense subset of $B$. 
A classic theorem in topology known as Engelking or Taimanov theorem 
says that a map to a compact Hausdorff space $K$ always extends from a dense subset $A$ to the whole domain $B$,
i.e.~$A\lra B \rtt K\to \verb|{o}|$; in fact the proof of this theorem also gives that 
$A\lra B \rtt K_1\lra K_2$ holds for any proper map $K_1\lra K_2$ of normal Hausdorff spaces. 
See [mintsGE,\S2.2] for a discussion. 

4. Each map in $(P)$ is both surjective and proper, and both 
being surjective and being proper are defined by right lifting properties. This implies that
each map in $(P)$ is both surjective and proper. 

Now let $X\lra Y$ be a surjective proper map of compact Hausdorff spaces. 
We need to show that it is in $(P)^\lr$, 
i.e.~that for each $A\lra B \in (P)^\lrl$
it holds $ A\lra B \rtt X\lra Y$. 

We know that $A\lra B$ lifts with respect to each map in $P$, hence by Lemma~\ref{lem:lrlp}
we may assume that $A$ is an open subset of $B$, and the map is the inclusion. 

Let $\bar A=\operatorname{Cl}\operatorname{Im}_B(A)$ be the closure of $A$ in $B$.

Now consider the lifting property \verb$A-->B /_ {u->a,b<-v}-->{u->a=b<-v}$ 
defining extremally disconnected.  Take \verb$A-->{u->a,b<-v}$ taking $A$ to $u$,
and \verb$B-->{u->a=b<-v}$ sending $A$ to \verb$u$, and $\bar A\setminus A$ to \verb$a=b$,
and $B\setminus \bar A$ to \verb$v$. If $A$ is non-empty, the lifting property implies 
that $\bar A$ is open.  Hence, both  $\bar A$ 
and $B\setminus \bar A$ are closed open subsets, and 
to construct the diagonal map, it is enough to construct it separately 
on $\bar A$ and $B\setminus \bar A$. As $A$ is dense in $\bar A$, by Lemma above 
$A\to \bar A$ lifts with respect to any proper map of compact Hausdorff spaces;
this implies the former. 

Note that \verb$A-->B /_ {u->a,b<-v}-->{u->a=b<-v}$
implies that $$\emptyset \lra B\setminus \bar A\rtt \verb|{u->a,b<-v}-->{u->a=b<-v}|$$
i.e.~that $ B\setminus \bar A$ is extremally disconnected.
Finally, the theorem of Gleason that extremally disconnected sets are projective 
in the subcategory of compact Hausdorff spaces with proper maps, 
says precisely that this lifting property holds for each surjective proper map of compact Hausdorff spaces. 
\end{proof}
\begin{rema} We rely on the Gleason theorem rather than reproduce its proof. 
Probably a careful reformulation of Lemmas~2.1 and 2.4 of [Gleason] shall turn the proof there
	into a diagram chasing calculation with finite preorders.  
\end{rema}

\subsection{Being a surjective image of a compact extremally disconnected space}

We start with the observation that the map \verb|{u->a,b<-v}-->{u->a=b<-v}| 
appearing in the definition of extremally disconnected, is surjective and proper, and that being surjective and being proper 
are right Quillen negations.  
Hence, each map in $(P_0)^\lr$ is both surjective and proper, 
where $P_0:=\{\verb|{u->a,b<-v}-->{u->a=b<-v}|\}$ is the class
consisting of a single morphism \verb|{u->a,b<-v}-->{u->a=b<-v}|.
Moreover, note 
$\verb|{o->c}-->{o=c} /_ {u->a,b<-v}-->{u->a=b<-v}|$.
This implies that if each point of $Y$ is closed (i.e.~$Y$ satisfies separation axiom $T1$)
and $X\lra Y$ is in $(P_0)^\lr$, then each point of $X$ is closed. 
For compact spaces, axiom $T1$ implies $T2$ (being Hausdorff). Hence,
if $X\lra Y$ is in $(P_0)^\lr$ and 
if $Y$ is compact Hausdorff, so is $X$.


Hence:
\begin{enonce}{Observation}\label{obs:1}
The fact that each topological space admits a surjection from an extremally disconnected space, 
and, moreover, each compact Hausdorff topological space admits a surjection from an 
compact Hausdorff extremally disconnected space, is implied by the following. 

Each morphism $\emptyset \lra X$ decomposes as 
$\emptyset\xrightarrow{(P_0)^\lrl} E \xrightarrow{(P_0)^\lr} X  $
where $P_0:= \{\verb|{u->a,b<-v}-->{u->a=b<-v}|\}$ is a class of morphisms consisting of a single morphism which is both surjective and proper (and hence so is any map 
in $(P_0)^\lr$).
\end{enonce}

In fact, this decomposition (weak factorisation system) implies that
the extremally disconnected subspaces are projective in the (not full!) 
subcategory of topological spaces with morphisms in $(P_0)^\lr$, 
and that subcategory has enough projectives. 

Unfortunately, not each surjective proper map of compact Hausdorff spaces is in
 $(P_0)^\lr$. Indeed, if the domain is connected, then it maps to one of the connected components
 \verb|{u->a}| or \verb|{b<-v}|, and by surjectivity the codomain does as well.
 Hence, any surjective map from a connected space is in $(P_0)^\lrl$
 and thus not in $(P_0)^\lr$ unless is an isomorphism.

\subsection{Extremally disconnected spaces being projective} 
Let $P'$ denote the class of
all the closed (necessarily proper) surjective 
maps mentioned 
in right lifting properties in Lemma~\ref{lem:lrlp} (we give the list of morphisms in various notations): 
$$\hskip-42pt\begin{array}{ccccc}
	\{u\rotatebox{-12}{\ensuremath{\to}}\raisebox{-2pt}{\ensuremath{a\,\,,\hskip-6pt\raisebox{3pt}{\bf\color{red}=}b}\rotatebox{13}{\ensuremath{\leftarrow}}} u\} & 
\{a \leftrightarrow \!\!\!\!\!\!\raisebox{6pt}{{\color{red}\bf\!\!=\,\,}} b\}& 
\{a \rightarrow \!\!\!\!\!\!\raisebox{6pt}{{\color{red}\bf\!\!=\,\,}} b\}& 
	\{\raisebox{0pt}{\ensuremath{a\llrra b}}\raisebox{6pt}{\color{red}\bf\,=}\!\!\!\!\!\!\rotatebox{-12}{\ensuremath{\to}}  \raisebox{-2pt}{c}\}
\\
	\{u\rotatebox{-12}{\ensuremath{\to}}\raisebox{-2pt}{\ensuremath{a,b}\rotatebox{13}{\ensuremath{\leftarrow}}} u\}\lra \{u\rotatebox{-12}{\ensuremath{\to}}\raisebox{-2pt}{\ensuremath{a=b}\rotatebox{12}{\ensuremath{\leftarrow}}}v\} &
\{a\llrra b\}\lra \{a=b\} & \{o\ra c\}\lra \{o=c\} & \{a\llrra b\ra c\}\lra \{a\llrra b=c\} &
\\	\verb|{u->a,b<-v}-->{u->a=b<-v}|& \verb|{a<->b}-->{a=b}|& \verb|{o->c}-->{o=c}|& \verb|{a<->b->c}-->{a<->b=c}|&
\\\text{(extremally disconnected)}        & \text{(injective)}            & \text{(pullback topology)}  & \text{(closed map and pullback topology)}&
\end{array}$$
We can combine together the latter three morphisms and take instead e.g.
$$\hskip-0pt\begin{array}{ccccc}
P'':=\{&  \verb|{u->a,b<-v}-->{u->a=b<-v}|,& \verb|{a<->b->c<->d}-->{a<->b=c=d}|&\}
\end{array}$$


%
%

We summarise the considerations above as
\begin{enonce}{Observation}\label{obs:2}
The fact that extremally disconnected space are projective in the category of compact Hausdorff spaces with proper maps, and this category has enough projectives, 
is implied by the decomposition above and the following. 

Each morphism $\emptyset \lra X$ decomposes as 
$\emptyset\xrightarrow{(P')^\lrl} E \xrightarrow{(P')^\lr} X  $
where 
$$P':=\{\verb|{u->a,b<-v}-->{u->a=b<-v}, {a<->b}-->{a=b}, {a<->b->c}-->{a<->b=c}|\}
$$ 
is a class of morphisms consisting of  surjective proper morphisms. 
\end{enonce}
\begin{proof} Use Lemma~\ref{lem:surjpr}(3). Use the observation above
to construct a compact Hausdorff $E$ fitting the decomposition. 
We may omit \verb${o->c}-->{o=c}$
because the map gluing together \verb$a$ and \verb$b$ in 
\verb${a<->b->c}-->{a<->b=c}$ gives \verb${o->c}-->{o=c}$.
\end{proof}

\begin{rema} 
\,[Analytic,p7] writes ``for part (2) [the sheaf condition on condensed sets represented by topological spaces] 
the key point is that any surjective map of profinite sets is a quotient''. In fact any surjective proper map is a quotient,
and Lemma~\ref{lem:lrlp} and ~\ref{lem:surjpr} allows to express this as:
$$ P'^\lr\subset \{\verb|{o->c}-->{o<->c}|\}^\lrl $$ 
Explicitly, 
$$\hskip-42pt\{\verb|{u->a,b<-v}-->{u->a=b<-v}, {a<->b}-->{a=b}, {a<->b->c}-->{a<->b=c}|\}^\lr \subset 
\{\verb|{o->c}-->{o<->c}|\}^\lrl $$
$$\hskip-42pt\left\{\{u\rotatebox{-12}{\ensuremath{\to}}\raisebox{-2pt}{\ensuremath{a,b}\rotatebox{13}{\ensuremath{\leftarrow}}} v\}\lra \{u\rotatebox{-12}{\ensuremath{\to}}\raisebox{-2pt}{\ensuremath{a=b}\rotatebox{12}{\ensuremath{\leftarrow}}}v\},
\{a\llrra b\rotatebox{-12}{\ensuremath{\to}}\raisebox{-2pt}{\ensuremath{c\llrra d}}\}\lra 
\{a\llrra b=c=d\}\right\}^\lr \subset 
\left\{\{o\rotatebox{-12}{\ensuremath{\to}}\raisebox{-2pt}{\ensuremath{c}}\}\lra \{o\llrra c\}\right\}^\lrl$$
\label{rema:quot}\end{rema}

\section{Appendix A. A list of reformulations of topological definitions}

Here we give a list of examples of iterated lifting properties (negations) written in our notation.
Sometimes we skip $\rtt$ for readability. Most of the list below is taken from [mintsGE, \S5.2].

\subsection{\label{app:rtt-top}Examples of iterated orthogonals obtained from maps between finite topological spaces.}
Here we give a list of examples of iterated orthogonals starting from maps between finite topological spaces 
defining well-known properties of topological spaces.

In the category of topological spaces,
\begin{enumerate}

\item        $(\emptyset\longrightarrow \{o\})^r$   is the class of surjections
\item        $(\emptyset\longrightarrow \{o\})^r$   is the class of maps $A\lra B$ where $A\neq \emptyset$ or $A=B$
\item        $(\emptyset\longrightarrow \{o\})^{rr}$ is the class of subsets, i.e. injective maps $A\hookrightarrow B$ where the topology on $A$ is induced from $B$

\item $(\emptyset\longrightarrow \{o\})^{lr}$ is the class of maps $\emptyset\longrightarrow B$, $B$ arbitrary
\item        $(\emptyset\longrightarrow \{o\})^{lrr}$ is the class of maps $A\longrightarrow B$ which admit a section

\item  $(\emptyset\longrightarrow \{o\})^l$ consists of maps $f:A\longrightarrow B$ such that either  $A\neq \emptyset$ or  $A=B=\emptyset$

\item $(\emptyset\longrightarrow \{o\})^{rl}$ is the class of maps of form $A\longrightarrow A\sqcup D$ where $D$ is discrete

\item $\{ \{z\llrra x\llrra y\ra c\}\lra\{z=x\llrra y=c\} \}^\lrl  = \{\{c\}\lra \{o\ra c\}\}^\lr$ is the class of closed inclusions $A\subset B$ where $A$ is closed

\item $\{ \{z\llrra x\llrra y\la c\}\lra\{z=x\llrra y=c\} \}^\lrl$ is the class of open inclusions
$A\subset B$ where $A$ is open

\item $\{ \{x\llrra y\ra c\}\lra\{x\llrra y=c\} \}^\lrl $ is the class of closed maps $A\lra B$ 
where the topology on $A$ is pulled back from $B$
\item $\{ \{x\llrra y\la c\}\lra\{x\llrra y=c\} \}^\lrl$ is the class of open maps $A\lra B$ 
where the topology on $A$ is pulled back from $B$

\item        $(\{b\}\longrightarrow \{a{\small\searrow}b\})^l$ is the class of maps with dense image
\item        $(\{b\}\longrightarrow \{a{\small\searrow}b\})^{lr}$ is the class of closed subsets $A \subset  X$, $A$ a closed subset of $X$
\item $( \{a{\small\searrow}b\}\lra\{a=b\})^l$ is the class of injections

\item        $((\{a\}\longrightarrow \{a{\small\searrow}b\})^r_{<5})^{lr}$ is roughly the class of proper maps
 


\end{enumerate}

\subsection{\label{app:rtt-top-prop}Examples of properties of topological spaces expressed as iterated orthogonals of maps between finite topological spaces.}
Here give a list of examples of well-known properties defined by 
iterated orthogonals starting from maps between finite topological spaces, often with less than 5 elements. 

\begin{enumerate}

\item $\{\bullet\}\lra A$ is in $(\emptyset\longrightarrow \{o\})^{rll}$ iff $A$ is connected
\item 
$Y$ is totally disconnected iff $\{\bullet\}\xra y Y$ is in $(\emptyset\longrightarrow \{o\})^{rllr}$ for each map  $\{\bullet\}\xra y Y$ (or, 
in other words, each point $y\in Y$).

\item  a Hausdorff space $K$ is compact iff $K\longrightarrow \{o\}$ is in  $((\{o\}\longrightarrow \{o{\small\searrow}c\})^r_{<5})^{lr}$
\item  a  Hausdorff space $K$ is compact iff $K\longrightarrow \{o\}$ is in  $$
     \{\, \{a\leftrightarrow b\}\longrightarrow \{a=b\},\, \{o{\small\searrow}c\}\longrightarrow \{o=c\},\,
     \{c\}\longrightarrow \{o{\small\searrow}c\},\,\{a{\small\swarrow}o{\small\searrow}b\}\longrightarrow \{a=o=b\}\,\,\}^{lr}$$

\item  a space $D$ is discrete iff $ \emptyset \longrightarrow  D$ is in $   (\emptyset\longrightarrow \{o\})^{rl}      $

\item  a space $D$ is antidiscrete iff $ \ensuremath{D} \longrightarrow  \{o\} $ is in 
$(\{a,b\}\longrightarrow \{a=b\})^{rr}= (\{a\leftrightarrow b\}\longrightarrow \{a=b\})^{lr} $ 

\item  a space $K$ is connected or empty iff $K\longrightarrow \{o\}$ is in  $(\{a,b\}\longrightarrow \{a=b\})^l $
\item  a space $K$ is totally disconnected and non-empty iff $K\longrightarrow \{o\}$ is in  $(\{a,b\}\longrightarrow \{a=b\})^{lr} $

\item  a space $K$ is connected and non-empty
 iff
 for some arrow $\{o\}\longrightarrow K$\\
$\text{ \ \ \ \ \     } \{o\}\longrightarrow K$ is in
            $   (\emptyset\longrightarrow \{o\})^{rll} = (\{a\}\longrightarrow \{a,b\})^l$

\item  a space $K$ is non-empty iff $K\longrightarrow \{o\}$ is in $   (\emptyset\longrightarrow \{o\})^l$
\item  a space $K$ is empty iff $K \longrightarrow \{o\}$ is in $   (\emptyset\longrightarrow \{o\})^{ll}$
\item  a space $K$ is $T_0$ iff $K  \longrightarrow \{o\}$ is in $   (\{a\leftrightarrow b\}\longrightarrow \{a=b\})^r$
\item   a space $K$ is $T_1$ iff $K  \longrightarrow \{o\}$ is in $   (\{a{\small\searrow}b\}\longrightarrow \{a=b\})^r$
\item  a space $X$ is Hausdorff iff for each injective map $\{x,y\} \hookrightarrow  X$
it holds $\{x,y\} \hookrightarrow  \ensuremath{X} \,\rightthreetimes\,  \{ \ensuremath{x} {\small\searrow} \ensuremath{o} {\small\swarrow} \ensuremath{y} \} \longrightarrow  \{ x=o=y \}$

\item  a non-empty space $X$ is regular (T3) iff for each arrow $    \{x\} \longrightarrow  X$ it holds
    $    \{x\} \longrightarrow  \ensuremath{X} \,\rightthreetimes\,  \{x{\small\searrow}X{\small\swarrow}U{\small\searrow}F\} \longrightarrow  \{x=X=U{\small\searrow}F\}$
\item  a space $X$ is normal (T4) iff $\emptyset \longrightarrow \ensuremath{X} \,\rightthreetimes\,   \{a{\small\swarrow}U{\small\searrow}x{\small\swarrow}V{\small\searrow}b\}\longrightarrow \{a{\small\swarrow}U=x=V{\small\searrow}b\}$

\item  a space $X$ is completely normal iff $\emptyset\longrightarrow \ensuremath{X} \,\rightthreetimes\,  [0,1]\longrightarrow \{0{\small\swarrow}x{\small\searrow}1\}$
 where the map $[0,1]\longrightarrow \{0{\small\swarrow}x{\small\searrow}1\}$ sends $0$ to $0$, $1$ to $1$, and the rest $(0,1)$ to $x$

\item a space $X$ is hereditary normal iff 
$ \emptyset \to X \rightthreetimes 
\{ x \la au \leftrightarrow u' \la u \la uv \ra v \ra v'\leftrightarrow bv \ra x \} 
\longrightarrow
\{ x \la au \leftrightarrow u' = u \la uv \ra v = v'\leftrightarrow bv \ra x \} $

\item  a space $X$ is path-connected iff $\{0,1\} \longrightarrow  [0,1] \,\rightthreetimes\,  \ensuremath{X} \longrightarrow  \{o\}$
\item  a space $X$ is path-connected iff for each Hausdorff compact space $K$ and each injective map $\{x,y\} \hookrightarrow  K$ it holds
   $\{x,y\} \hookrightarrow  \ensuremath{K} \,\rightthreetimes\,  \ensuremath{X} \longrightarrow  \{o\}$

\end{enumerate}

\subsection{A sample of a computer syntax} Here we rewrite some of the examples above in a computer syntax. 
ASCII art on the right attempts to represent graphically the maps of preorders involved. 
\begin{verbatim}
compactness:   { {o}-->{o->c} }^r_{<5}^lr  ;  {o}-->{o->c} is a non-proper map
dense image:  { {c}-->{o->c} }^l   ; the image of {o}-->{o->c} is not dense
injection:       { {x,y}-->{x=y} }^r == { {x<->y}-->{x=y} }^l   '~'(' == .-.).
surjection:       { {}-->{o} }^r == { {}-->{o} }^rrl  simplest non-surjection  {}-->{o}
connected:       { {}-->{o} }^rll   { {x,y}-->{x=y} }^l  simplest non-connected space {x,y}     
discrete:    { {}-->{o} }^rl             )(.  
subset:    { {}-->{o} }^rr == {{x<->y->c}-->{x=y=c}}^l  )). ==  ~\(.  
closed subset:     { {z<->x<->y->c}-->{z=x<->y=c} }^l == {{c}-->{o->c}}^lr 
open subset:       { {z<->x<->y<-c}-->{z=x<->y=c} }^l  '~'~'\ ( '='~'=. 
normal (T4):          { {a<-b->c<-d->e}-->{b=c=d} }^l    /V\(/\ 
Tietze lemma  (not quite) : R-->{o} (- { {a<-b->c<-d->e}-->{b=c=d},{a<-b->c}-->{a=b=c} }^lr 
Urysohn lemma (not quite): R-->{a<-b->c} (- { {a<-b->c<-d->e}-->{b=c=d} }^lr  
Hausdorff:   {u,v}--(inj)-->X /_ {u->x<-v}-->{u=x=v} 
             i.e. any injection {a,b}-->X lifts wrt {u->x<-v}-->{u=x=v}
\end{verbatim}             
%

\subsubsection*{Avoiding repetitions} The reader would notice that the syntax above repeats almost everyting twice:
indeed, almost the same preorder appears on both sides of \verb|-->| arrow. Below we give a sample of possible notations 
avoiding this repeation, hence the notation below is intentionally not consistent. 

\begin{Verbatim}[commandchars=|\{\}]
compactness:   |{ |{o.|textcolor{red}{->c}.|} |}^r_|{<5|}^lr  ;  |{o|}-->|{o|textcolor{red}{->c}|} is a non-proper map
dense image:  |{ |{.|textcolor{red}{o->}.c|} |}^l   ; the image of |{o|}-->|{|textcolor{red}{o->}c|} is not dense
injection:       |{ |{x,.|textcolor{red}{=}.y|} |}^r == |{ |{x<->y|}-->|{x|textcolor{red}=y|} |}^l   '~'(' == .-.).
surjection:       |{ |{.|textcolor{red}o.|} |}^r == |{ |{|}-->|{o|} |}^rrl  simplest non-surjection  |{|}-->|{|textcolor{red}o|}
connected:       |{ |{.|textcolor{red}o.|} |}^rll   |{ |{x,.|textcolor{red}=.y|} |}^l  simplest non-connected space |{x,y|}     
discrete:    |{ |{.|textcolor{red}o.|} |}^)(             )(.  
subset:    |{ |{.|textcolor{red}o.|} |}^)) == |{|{x<|textcolor{red}=>y|textcolor{red}=>c|}|}^(  )). ==  ~\(.  
closed subset:     |{ |{z<=>x<->y=>c|} |}^l == |{|{o.|textcolor{red}{->c}.|}|}^lr 
open subset:       |{ |{z<=>x<->y<=c|} |}^l  '~'~'\ ( '='~'=. 
normal (T4):          |{ |{a<-b=>c<=d->e|} |}^l    /V\(/\ 
Tietze lemma  (not quite) : R-->|{o|} (- |{ |{a<-b=>c<=d->e|}, |{a<=b=>c|} |}^() 
Urysohn lemma (not quite): R-->|{a<-b->c|} (- |{ |{a<-b=>c<=d->e|} |}^lr  
Hausdorff:   |{u,v|}--(inj)-->X /_ |{u=>x<=v|} 
             i.e. any injection |{a,b|}-->X lifts wrt |{u.=.>x<.=.v|}
\end{Verbatim}

\section{Appendix B (unfinished)}\label{sec:lemma13}

In this appendix (not indented for publication) we experiment with notation for diagram chasing calculations.
We present an incomplete(!) diagram chasing calculating representing the proof of Lemma 1.3 [Analytic].
We hope our calculations give some evidence that it may be possible to use diagram chasing with preorders 
in an efficient formalisation of general topology.

\subsection{Statement and proof of Lemma 1.3} We quote [Analytic]:
\begin{quote}
Lemma 1.3. {\em Let $X_0\lra X_1 ...$ and $Y_0 \lra  Y_1 \lra ...$ be two sequences of compact Hausdorff
spaces with closed immersions. Then, inside the category of topological spaces, the natural map
	$$\bigcup_n X_n \times Y_n \lra (\bigcup_n X_n)\times (\bigcup Y_n)$$
is a homeomorphism; i.e.~the product on the right is equipped with its compactly generated topology.}
\begin{proof} The map is clearly a continuous bijection. In general, for a union like 
	$\cup_n X_n$, open
subsets $U$ are the subsets of the form $\cup_n  U_n$ where each $U_n \subset X_n$ is open. 
	Thus, let $U \subset \cup_n X_n\times Y_n$
be any open subset, written as a union of open subset $U_n \subset X_n \times Y_n$, and pick any point $(x,y) \in U$.
Then for any large enough $n$ (so that $(x,y) \in X_n \times Y_n)$, we can find open neighborhoods $V_n \subseteq X_n$
of $x$ in $X_n$ and $W_n  \subseteq Y_n$ of $y$ in $Y_n$, such that $V_n \times W_n  \subseteq U_n$. 
In fact, we can ensure that even
$\bar V_n \times \bar W_n  \subseteq U_n$ by shrinking $V_n$ and $W_n$. Constructing the $V_n$ and $W_n$  inductively, we may then
moreover ensure $V_ n \subseteq V_{n+1}$ and $W_n  \subseteq W_{n +1}$. Then $V = \bigcup_n V_n \subseteq \bigcup_n X_n$ and 
$W = \bigcup_ n W_n  \subseteq\bigcup_ 
n Y_n$ are open, and $V \times W = \bigcup_n V_n \times W_n  \subseteq U$ contains $(x,y)$, showing that $U$ is open in the
product topology.\end{proof}
\end{quote}

\subsubsection{Partially commutative diagrams: $@\{o\}$.} 
In a computation it is useful to consider partially commutative diagrams and we extend our lifting property notation 
accordingly.  Given a diagram, and 
a letter \verb$o$, possibly occurring in notation of one of the finite preorders (topological spaces), 
and an arrow $X\lra Y$ in the diagram, 
we label it by \verb$@{o}$ as $X\xra{@\{o\}} Y$ to indicate that we only care about commutativity requirements
with respect to elements {\em denoted} by \verb$o$.
In notation, we say that two paths 
$X=X_1 \xra{f_{1}}  ...  \xra{f_{k-1}}  X_k \xra[f_k]{@\{o\}} X_{k+1} \xra{f_{l+1}} X_{k+2}  \xra{f_{k+}}... \xra{f_{k'}}X_{k'+1}=Y$ and   
$Y=Y_1 \xra{g_{1}}  ...  \xra{g_{l-1}}  Y_k \xra{g_l} Y_{l+1} \xra{g_{l+1}} Y_{l+2}  \xra{f_{l+3}}... \xra{f_{l'}}Y_{l'+1}=Y$ 
commute iff both $f_{k'}(f_{k'-1}(..f_1(o)...)=g_{l'}(g_{l'-1}(...g_1(o)...))$ whenever $X$ has a point denoted by $o$,
and 
 $f_{k'}(f_{k'-1}(..f_1(x)...)=g_{l'}(g_{l'-1}(...g_1(x)...))=o$ whenever $Y$ has a point denoted by $o$ and
 $f_{k'}(f_{k'-1}(..f_1(x)...)=o$.

\subsection{Expanding the colimits} The first step in the proof is to  ``expand'' the colimits and get a diagram without colimits. 
(1)
We need to show that the topology on $\bigcup_n X_n \times Y_n$ is induced from $(\bigcup_n X_n)\times (\bigcup_n Y_n)$, 
i.e. the lifting  property
$$\bigcup_n X_n \times Y_n \lra(\bigcup_n X_n)\times (\bigcup_n Y_n) \rtt \verb. {o->c}-->{o=c}.$$

(2) Being open and being commutative is defined pointwise, hence it is enough to construct 
the following diagram for each point 
$\{o\}\xra{(x,y)} \bigcup_n X_n \times Y_n $ 
$$\xymatrix{ \{o\} \ar[r]^{(x,y)} & \bigcup_n X_n \times Y_n \ar[r] \ar@{->}[d] & \{o\ra c\} \ar[d] 
\\ {} & (\bigcup_n X_n)\times (\bigcup_n Y_n)  \ar[r]|-{} \ar@{-->}[ur]|{@\{o\}} & \{o=c\} }$$

(3) Expanding the definition of product topology we see it is enough to construct the following diagram 
$$\hskip-42pt \xymatrix{ 
\{o\} \ar[r]^{(x,y)} & \bigcup_n X_n \times Y_n \ar[r]|{\text{}U} \ar@{->}[d]|f & \{vw=o\ra c=vw'=v'w=v'w'\} \ar[d]|g 
\\ {} & (\bigcup_n X_n)\times (\bigcup_n Y_n) \ar[dl] \ar[dr]	\ar@{-->}[ddd] \ar[r]|-{j} \ar@{-->}[ur]|{@\{o\}}& \{o=c\} 
\\  (\bigcup_n X_n) \ar@{-->}[d] & & (\bigcup_n Y_n) \ar@{-->}[d] 
\\ \{v\ra v'\} \ar[dr] & &   \{w\ra w'\} \ar[dl]
\\ & \{v\ra v'\}\times\{w\ra w'\}=\{vw\ra vw'\ra v'w', vw\ra v'w \ra v'w'\} \ar[uuuur]|{@\{o\}}
}$$

(4) To construct an arrow from the union/colimit one needs to construct compatible arrows from each $X_n$ and $Y_n$. 
In a diagram chasing computation, we may do so by showing 
the inductive step that given an arrow from $X_n$, you can always add an arrow from $X_{n+1}$
fitting into the same diagram, and the same for the $Y$'s. 
In fact we may label the arrows from $X_n$ and $Y_n$ by \verb$@{o}$: in natural language this means
we are constructing an increasing sequence of open subsets.
We may assume $(x,y)\in X_n\times Y_n$. 

It is sufficient to construct the following diagram. 
$$\hskip-42pt\hskip-42pt
 \xymatrix{ 
&\{o\} \ar[r]^{(x,y)} \ar[d]\ar[drr] & \bigcup_n X_n \times Y_n \ar[r]|{\text{}} \ar@{->}[d] & \{vw=o\ra c=vw'=v'w=v'w'\}  
\\ ...\longleftarrow X_{n+1}\ar[rd]\ar@{-->}[rdd]\ar@{<-}[r] &{X_n}\ar[d]\ar@/^3pc/@{->}[dd]|{@\{o\}} & (\bigcup_n X_n)\times (\bigcup_n Y_n) \ar[dl] \ar[dr]	\ar@{-->}[ddd]  
\ar@{-->}[ur]& Y_n\ar[d]\ar[r]\ar@/_3pc/@{->}[dd]|{@\{o\}} & Y_{n+1}\lra...\ar[ld]\ar@{-->}[ldd]
\\&  (\bigcup_n X_n)  & & (\bigcup_n Y_n) 
\\& \{o=v\ra v'\} \ar@{<-}[dr] & &   \{o=w\ra w'\} \ar@{<-}[dl]
\\& & \{v\ra v'\}\times\{w\ra w'\}=\{o=vw\ra ...\}
\ar[uuuur]|{@{\{o\}}}
}$$

(4) We now add to the diagram the products $X_n\times Y_n$ and $X_{n+1}\times Y_{n+1}$, and remove the product of colimits. 
$$\hskip-42pt
 \xymatrix{ 
&\{o\} \ar[r]^{(x,y)} \ar[d]\ar[dr]\ar[drr]  & \bigcup_n X_n \times Y_n \ar[r]|{} \ar@{<-}[d] & \{vw=o\ra c=vw'=v'w=v'w'\}  
\\ ...\longleftarrow X_{n+1}\ar[rd]\ar@{-->}[rdd]\ar@{<-}[r] &{X_n}\ar[d]\ar@/^3pc/@{->}[dd]|{@\{o\}} & X_n\times  Y_n \ar@<0.2mm>[r]\ar[r] \ar@<0.2mm>[l]\ar[l] \ar[dl] \ar[dr]	\ar[d]  
\ar[ur]& Y_n\ar[d]\ar[r] \ar@/_3pc/@{->}[dd]|{@\{o\}} & Y_{n+1}\lra...\ar[ld]\ar@{-->}[ldd]
\\&  (\bigcup_n X_n)  & X_{n+1}\times Y_{n+1} \ar[dd] \ar[ull] \ar[urr]  \ar[uur]& (\bigcup_n Y_n) 
\\& \{o=v\ra v'\} \ar@{<-}[dr] & &   \{o=w\ra w'\} \ar@{<-}[dl]
\\& & \{v\ra v'\}\times\{w\ra w'\}
\ar[uuuur]|{@\{o\}}
}$$

(6) Now remove more vertices no longer needed:
$$\hskip-42pt\small\hskip-31pt
\xymatrix{ 
&\{o\} \ar[rr]^{(x,y)} \ar[d]\ar[dr]\ar[drr]  &  & \{vw=o\ra c=vw'=v'w=v'w'\}  
\\ ...\longleftarrow X_{n+1}\ar@{-->}[rdd]\ar@{<-}[r] &{X_n}\ar@/^3pc/@{->}[dd]|{@\{o\}} 
& X_n\times  Y_n \ar@<0.2mm>[r]\ar[r] \ar@<0.2mm>[l]\ar[l] 	\ar[d]  
\ar[ur]& Y_n\ar[r] \ar@/_3pc/@{->}[dd]|{@\{o\}} & Y_{n+1}\lra...\ar@{-->}[ldd]
\\&    & X_{n+1}\times Y_{n+1} \ar[dd] \ar[ull] \ar[urr]  \ar[uur]& 
\\& \{o=v\ra v'\} \ar@{<-}[dr] & &   \{o=w\ra w'\} \ar@{<-}[dl]
\\& & \{v\ra v'\}\times\{w\ra w'\}
\ar[uuuur]|{@\{o\}}
}$$

\subsubsection{Preliminary lemmas}
Recall that a topological space $X$ is normal (T4)
if any two disjoint closed subsets of $X$ are separated by neighbourhoods,
or, equivalently by Urysohn lemma, by a continuous function to $\Bbb R$.

Recall that a topological space $X$ is regular (T3) if, given any point $x\in X$ and 
closed set $B$ in $X$ such that $x$ does not belong to $B$, they are separated by neighbourhoods,
or, equivalently, by closed neighbourhoods. 

We do use the next Lemma and give it only for context. 

\begin{lemm} 
A topological space $X$ is normal (T4) iff either of the following equivalent conditions holds:
\begin{itemize}
\item $\emptyset\lra X \rtt \verb|{a<-v->x<-w->b}-->{a<-v=x=w->b}|$
\item $\emptyset\lra X \rtt \verb|{a<-v->v'<-x->w'<-w->b}-->{a<-v=v'=x=w'=w->b}|$
\end{itemize}

A topological space $X$ is regular (T3) iff 
$\verb|{v}|\lra X \rtt \verb|{v->a<-w->b}-->{v=a=w->b}|$
\end{lemm}
\begin{proof}
The preimages of \verb$a$ and \verb$b$ are disjoint closed subsets of $X$; 
the preimages of \verb${a<-v}$ and \verb${w->b}$ are open neighbourhoods separating $A$ and $B$.

	Let us consider the only interesting case is when \verb$v$ maps to \verb$v$ in \verb|{v=a=w->b}| by 
	the top horizontal arrow.
	The image of \verb$v$ is a point $v\in X$; 
the preimage of \verb|b| is a closed subset $B\ni v$ of $X$ not containing $v$,
or, equivalently, the preimage of \verb|v=a=w| is an open neighbourhood of $v$.
The preimage $V$ of $\verb|v|\in \verb|{v->a<-w->b}|$ is an open neighbourhood of $v$ in $X$ 
disjoint from $B$.
\end{proof}

Recall that a neighbourhood of a point is a set containing an open subset containing the point.

\begin{lemm} In a compact Hausdorff space, a neighbourhood of a point contains a closed neighbourhood of the point,
	and this is expressed by the following lifting property:
	\begin{itemize}
		\item \verb|{o}--> X /_ {o->a<-u->c}-->{o=a=u->c}|
	\end{itemize}
\end{lemm}
\begin{proof} Indeed, the lifting property holds trivially if the top horizontal arrow maps \verb|o| into \verb|u| or \verb|a|,
	so consider the case it maps to \verb|o|. 
	The preimage of \verb|o=a=u| is an open neighbourhood $U$ of $\verb|o|\in X$. The preimage of 
	\verb|o->a| by the diagonal arrow is a closed neighbourhood of $o$ contained in $U$.
\end{proof}

\begin{lemm} A map $X\lra Y$ is a closed inclusion iff it can be obtained as a basechange from \verb${c}-->{o->c}$
along some map $Y\lra\verb|{o->c}|$.
\end{lemm}
\begin{proof} Obvious. \end{proof}

\subsubsection{Diagram chasing proof continued}
(7) Use that $Y_n$ is a closed subset of $Y_{n+1}$ by representing it as a pullback
of $y_n$ in $\{y'_{n}\ra y_n\}$. 
$$\hskip-12pt\small\hskip-31pt
\xymatrix{ 
&\{o\} \ar[rr]^{(x,y)} \ar[d]\ar[dr]\ar[drr]  &  & \{vw=o\ra c=vw'=v'w=v'w'\}  
\\ ...\longleftarrow X_{n+1}\ar@{-->}[rdd]\ar@{<-}[r] &{X_n}\ar@/^3pc/@{->}[dd]|{@\{o\}} 
& X_n\times  Y_n \ar@<0.2mm>[r]\ar[r] \ar@<0.2mm>[l]\ar[l] 	\ar[d]  
\ar[ur]& Y_n \ar[d] \ar[r] \ar@/_3pc/@{->}[dd]|{@\{o\}} & Y_{n+1}\lra...\ar@{-->}[ldd]\ar[d]
\\&    & X_{n+1}\times Y_{n+1} \ar[dd] \ar[ull] \ar[urr]  \ar[uur]&\{y_n\}\ar[r] & \{y'_n\ra y_n\}
\\& \{o=v\ra v'\} \ar@{<-}[dr] & &   \{o=w\ra w'\} \ar@{<-}[dl] 
\\& & \{v\ra v'\}\times\{w\ra w'\}
\ar[uuuur]|{@\{o\}}
}$$
(8) By normality of $Y_n$ construct $Y_n\lra\{o\ra \barw \la w \ra w'\}$, 
i.e.~an closed subset $y\in\bar W_n\subset W_n$ 
containing an open neighbourhood (the preimage of $o$) of $y$.
Here the subset $W_n$ (open in $Y_n$) is the preimage of $\{o\ra\barw \la w\}$,
and the closed subset $\bar W_n$ is the preimage of $\{o\ra\barw\}$.  
$$\hskip-42pt\small
\xymatrix{ 
&\{o\} \ar[rr]^{(x,y)} \ar[d]\ar[dr]\ar[drr]  &  & \{vw=o\ra c=vw'=v'w=v'w'\}  
\\ ...\longleftarrow X_{n+1}\ar@{-->}[rdd]\ar@{<-}[r] &{X_n}\ar@/^3pc/@{->}[dd]|{@\{o\}} 
& X_n\times  Y_n \ar@<0.2mm>[r]\ar[r] \ar@<0.2mm>[l]\ar[l] 	\ar[d]  
\ar[ur]& Y_n \ar[d] \ar[r] \ar@/_3pc/@{->}[dd]|{@\{o\}}\ar@/_5pc/[ddd] & Y_{n+1}\lra...\ar@{-->}[ldd]\ar[d]
\\&    & X_{n+1}\times Y_{n+1} \ar[dd] \ar[ull] \ar[urr]  \ar[uur]&\{y_n\}\ar[r] & \{y'_n\ra y_n\}
\\& \{o=v\ra v'\} \ar@{<-}[dr] & &   \{o=\barw = w\ra w'\} \ar@{<-}[dl] 
\\& & \{v\ra v'\}\times\{w\ra w'\}\ar[uuuur]|{@\{o\}}
 &  \{o\ra \barw \la w \ra w'\}\ar[u] 
}$$
(9) Use that $Y_n$ is a closed subset to construct a map $Y_{n+1}\lra\{y'_{n}\ra o\ra \barw \la w \ra w', y'_n\ra w\}$,
i.e.~consider $\bar W_n\subset W_n$ as  subsets of $Y_{n+1}$; here $Y_{n+1}\setminus Y_n$ is the preimage of $y'_n$.
$$\hskip-42pt\hskip-42pt\small
\xymatrix{ 
&\{o\} \ar[rr]^{(x,y)} \ar[d]\ar[dr]\ar[drr]  &  & \{vw=o\ra c=vw'=v'w=v'w'\}  
\\ ...\longleftarrow X_{n+1}\ar@{-->}[rdd]\ar@{<-}[r] &{X_n}\ar@/^3pc/@{->}[dd]|{@\{o\}} 
& X_n\times  Y_n \ar@<0.2mm>[r]\ar[r] \ar@<0.2mm>[l]\ar[l] 	\ar[d]  
\ar[ur]& Y_n \ar[d] \ar[r] \ar@/_3pc/@{->}[dd]|{@\{o\}}\ar[rdd] \ar@/_5pc/[ddd] & Y_{n+1}\lra...\ar@{-->}[ldd]\ar[d]\ar@/^3pc/[dd]
\\&    & X_{n+1}\times Y_{n+1} \ar[dd] \ar[ull] \ar[urr]  \ar[uur]&\{y_n\}\ar[r] & \{y'_n\ra y_n=w=w'=o\}
\\& \{o=v\ra v'\} \ar@{<-}[dr] & &   \{o=w=\barw\ra w'\} \ar@{<-}[dl] & \{y'_{n}\ra o\ra \barw \la w \ra w', y'_n\ra w\}\ar[u]  
\\& & \{v\ra v'\}\times\{w\ra w'\}
\ar[uuuur]|{@\{o\}} & \{o\ra \barw \la w \ra w'\}\ar[u] \ar[ur]
}$$
(10) By symmetry do the same for $X_{n+1}$, i.e.~find a closed subset $x\in\bar W_n\subset W_n$ of $X_n$ 
containing an open neighbourhood of $x$, and consider it as a subset of $X_{n+1}$.
$$\hskip-42pt\hskip-42pt\tiny
\xymatrix@C-=4.5cm{ 
&\{o\} \ar[rr]^{(x,y)} \ar[d]\ar[dr]\ar[drr]  &  & \{vw=o\ra c=vw'=v'w=v'w'\}  
\\ ...\longleftarrow X_{n+1}\ar[d]\ar@/_3pc/[dd]\ar@{-->}[rdd]\ar@{<-}[r] &{X_n}\ar@/^5pc/[ddd]\ar[d]\ar[ddl]\ar@/^3pc/@{->}[dd]|{@\{o\}} 
& X_n\times  Y_n \ar@<0.2mm>[r]\ar[r] \ar@<0.2mm>[l]\ar[l] 	\ar[d]  
\ar[ur]& Y_n \ar[d] \ar[r] \ar@/_3pc/@{->}[dd]|{@\{o\}}\ar[rdd] \ar@/_5pc/[ddd] & Y_{n+1}\lra...\ar@{-->}[ldd]\ar[d]\ar@/^3pc/[dd]
\\ \{x'_n\ra x_n=v=v'=o\}&  \{x_n\}\ar[l]  & X_{n+1}\times Y_{n+1} \ar[dd] \ar[ull] \ar[urr]  \ar[uur]&\{y_n\}\ar[r] & \{y'_n\ra y_n=w=w'=o\}
\\ \{x'_{n}\ra o\ra \barv \la v \ra v',, x'_n\ra v\}\ar[u]&   \{o=\barv=v\ra v'\} \ar@{<-}[dr] & &   \{o=\barw=w\ra w'\} \ar@{<-}[dl] & \{y'_{n}\ra o\ra \barw \la w \ra w'\}
\\ & \{o\ra \barv \la v \ra v'\}\ar[u] \ar[ul] & \{v\ra v'\}\times\{w\ra w'\}
\ar[uuuur]|{@\{o\}} 
& \{o\ra \barw \la w \ra w'\}\ar[u] \ar[ur]
}$$
(11) Finally, we constructed 
$(x,y)\in V_n\times W_n \subset \bar V_n\times \bar W_n\subset X_n\times Y_n\subset X_{n+1}\times Y_{n+1}$ and 
$\bar W_n\times \bar W_n\subset U$ where $V_n$ is open in $X_n$, and $W_n$ is open in $Y_n$, and  $\bar V_n$ is closed in both $X_n$ and $X_{n+1}$,
and so is $\bar W_n$ in $Y_n$ and $Y_{n+1}$.
$$\hskip-42pt\hskip-42pt\tiny
\xymatrix@C-=0.5cm{ 
&\{o\} \ar[rr]^{(x,y)} \ar[d]\ar[dr]\ar[drr]  &  & \{vw=o\ra c=vw'=v'w=v'w'\}  
\\ ...\longleftarrow X_{n+1}\ar@/_3pc/[dd]\ar@{-->}[rdd]\ar@{<-}[r]\ar[d] &{X_n}\ar@/^5pc/[ddd]\ar[d]\ar[ddl]\ar@/^3pc/@{->}[dd]|{@\{o\}} 
& X_n\times  Y_n \ar@<0.2mm>[r]\ar[r] \ar@<0.2mm>[l]\ar[l] 	\ar[d]  
\ar[ur]& Y_n \ar[d] \ar[r] \ar@/_3pc/@{->}[dd]|{@\{o\}}\ar[rdd] \ar@/_5pc/[ddd] & Y_{n+1}\lra...\ar@{-->}[ldd]\ar[d]\ar@/^3pc/[dd]
\\ \{x'_n\ra x_n=v=v'=o\}&  \{x_n\}\ar[l]  & X_{n+1}\times Y_{n+1} \ar[dd] \ar[ull] \ar[urr]  \ar[uur]&\{y_n\}\ar[r] & \{y'_n\ra y_n=w=w'=o\}
\\ \{x'_{n}\ra o\ra \barv \la v \ra v', x'_n\ra v\}\ar[u]&   \{o=\barv=v\ra v'\} \ar@{<-}[dr] & &   \{o=\barw=w\ra w'\} \ar@{<-}[dl] & \{y'_{n}\ra o\ra \barw \la w \ra w', y'_n\ra w\}
\\ & \{o\ra \barv \la v \ra v'\}\ar[u] \ar[ul] & \{v\ra v'\}\times\{w\ra w'\}
\ar[uuuur]|{@\{o\}} 
& \{o\ra \barw \la w \ra w'\}\ar[u] \ar[ur]
}$$
(12) Now we shall use that $X_{n+1}\times Y_{n+1} \lra Y_{n+1}$ is closed. 
We have a closed subset $\bar W_n \subset X_{n+1}$, 
and an open subset of $U\subset X_{n+1}\times Y_{n+1}$, and we know that $\bar V_n \times \bar W_n \subset U$. 
Using that the projection $X_{n+1}\times Y_{n+1} \lra Y_{n+1}$ is closed, we find 
an open subset $\bar W_n \subset W_{n+1}\subset Y_{n+1}$ (the preimage of \verb|o| in 
$Y_{n+1}\lra \{o=*o = x_n'c=vc=v'c\ra oc=\barv c\}$)
such that $\bar V_n \times \bar W_{n+1} \subset U$.
In this diagram, we use yet another extension of our  notation: \verb|@{o}| above the arrow
means (in this picture) that it relates only to the triangle above and not below.
$$\hskip-42pt\hskip-42pt\hskip-21pt\tiny
\xymatrix@C-=0.5cm{ 
&\{o\} \ar[rr]^{(x,y)} \ar[d]\ar[dr]\ar[drr]  &  & \{vw=o\ra c=vw'=v'w=v'w'\}  
\\ ...\longleftarrow X_{n+1}\ar@/_3pc/[dd]\ar@{-->}[rdd]\ar@{<-}[r]\ar[d] &{X_n}\ar@/^5pc/[ddd]\ar[d]\ar[ddl]\ar@/^3pc/@{->}[dd]|{@\{o\}} 
& X_n\times  Y_n \ar@<0.2mm>[r]\ar[r] \ar@<0.2mm>[l]\ar[l] 	\ar[d]  
\ar[ur]& Y_n \ar[d] \ar[r] \ar@/_3pc/@{->}[dd]|{@\{o\}}\ar[rdd] \ar@/_5pc/[ddd] & Y_{n+1}\lra...\ar@{-->}[ldd]\ar[d]\ar@/^3pc/[dd]
\\ \{x'_n\ra x_n=v=v'=o\}&  \{x_n\}\ar[l]  & X_{n+1}\times Y_{n+1} \ar[dd] \ar[ull] \ar[urr]  \ar[uur]&\{y_n\}\ar[r] & \{y'_n\ra y_n=w=w'=o\}
\\ \{x'_{n}\ra o\ra \barv \la v \ra v', x'_n\ra v\}\ar[u]&   \{o=\barv=v\ra v'\} \ar@{<-}[dr] & &   \{o=\barw=w\ra w'\} \ar@{<-}[dl] & \{y'_{n}\ra o\ra \barw \la w \ra w', y'_n\ra w\}
\\ & \{o\ra \barv \la v \ra v'\}\ar[u] \ar[ul] & \{v\ra v'\}\times\{w\ra w'\}
\ar[uuuur]|{@\{o\}} 
& \{o\ra \barw \la w \ra w'\}\ar[u] \ar[ur]
\\
\\ &\{x'_{n}\ra o\ra \barv \la v \ra v', x'_n\ra v\}\times \{o\ra c\}
\ar[uuul]\ar[uuuuuurr]\ar@/^2pc/@{<-}[uuuur]\ar[d] & & \{o=\barw \leftrightarrow w=w'=y'\}\ar@{<-}[uuur]
\\ &
\{o=*o = x_n'c=vc=v'c\ra oc=\barv c\}  \ar@{<-}[u] \ar[urr]|{@\{o\}} \ar@/_3pc/@{<--}[uuuuuurrr]^{@\{o\}} 
}$$
(13) Unfortunately, the last step is somewhat vague, as it requires a careful handling of the inductive step, 
which we are not able to do. Neither do we carefully specify the commutativity conditions, unfortunately.

Finally, by normality of $Y_{n+1}$ add an arrow $\{o=*o = x_n'c=vc=v'c\ra oc=\barv c\}  \xra{@\{o\}} \{o\ra \barw \la w \ra w'\}$.  We now see that our calculation shows how to add an arrow  
$Y_{n+1} \xra{@\{o\}} \{o\ra \barw \la w \ra w'\}$ given an arrow 
$Y_{n} \xra{@\{o\}} \{o\ra \barw \la w \ra w'\}$, fitting the same diagram. 
By symmetry we may do the same for $X_n$ and $X_{n+1}$, and this would complete the inductive step, and 
thereby the argument. Note that we modified the inductive step assumption from (7). 
$$\hskip-42pt\hskip-42pt\hskip-21pt\tiny
\xymatrix@C=0.5cm{ 
&\{o\} \ar[rr]^{(x,y)} \ar[d]\ar[dr]\ar[drr]  &  & \{vw=o\ra c=vw'=v'w=v'w'\}  
\\ ...\longleftarrow X_{n+1}\ar@/_3pc/[dd]\ar@{-->}[rdd]\ar@{<-}[r]\ar[d] &{X_n}\ar@/^5pc/[ddd]\ar[d]\ar[ddl]\ar@/^3pc/@{->}[dd]|{@\{o\}} 
& X_n\times  Y_n \ar@<0.2mm>[r]\ar[r] \ar@<0.2mm>[l]\ar[l] 	\ar[d]  
\ar[ur]& Y_n \ar[d] \ar[r] \ar@/_3pc/@{->}[dd]|{@\{o\}}\ar[rdd] \ar@/_5pc/[ddd] & Y_{n+1}\lra...\ar@{-->}[ldd]\ar[d]\ar@/^3pc/[dd]
\\ \{x'_n\ra x_n=v=v'=o\}&  \{x_n\}\ar[l]  & X_{n+1}\times Y_{n+1} \ar[dd] \ar[ull] \ar[urr]  \ar[uur]&\{y_n\}\ar[r] & \{y'_n\ra y_n=w=w'=o\}
\\ \{x'_{n}\ra o\ra \barv \la v \ra v', x'_n\ra v\}\ar[u]&   \{o=\barv=v\ra v'\} \ar@{<-}[dr] & &   \{o=\barw=w\ra w'\} \ar@{<-}[dl] & \{y'_{n}\ra o\ra \barw \la w \ra w', y'_n\ra w\}
\\ & \{o\ra \barv \la v \ra v'\}\ar[u] \ar[ul] & \{v\ra v'\}\times\{w\ra w'\}
\ar[uuuur]|{@\{o\}} 
& \{o\ra \barw \la w \ra w'\}\ar[u] \ar[ur]
\\
\\ &\{x'_{n}\ra o\ra \barv \la v \ra v', x'_n\ra v\}\times \{o\ra c\}
\ar[uuul]\ar[uuuuuurr]\ar@/^2pc/@{<-}[uuuur]\ar[d] & & \{o=\barw \leftrightarrow w=w'=y'\}\ar@{<-}[uuur]
\\ &
\{o=*o = x_n'c=vc=v'c\ra oc=\barv c\}  \ar@{<-}[u] \ar[urr]|{@\{o\}} \ar@{<-}[uuurr]|{o=\barw =w} \ar@/_3pc/@{<--}[uuuuuurrr]^{@\{o\}} 
}$$


\begin{thebibliography}{10}


\bibitem[Bourbaki]{Bourbaki}
\newblock Nicolas Bourbaki.
\newblock General Topology.
\newblock I\S10.2, Thm.1(d), p.101 (p.106 of file)
\href{http://mishap.sdf.org/mints-lifting-property-as-negation/tmp/Bourbaki_General_Topology.djvu}{General Topology. I\S10.2, Thm.1(d), p.101 (p.106 of file)}

\bibitem[Engelking]{Engelking}
\newblock Ryszard Engelking.
\newblock General Topology. Thm.3.2.1, p.136.



\bibitem[Gavrilovich, DMG]{Gavrilovich, DMG}
\newblock Misha Gavrilovich.
\newblock Point set topology as diagram chasing computations. Lifting properties as intances of  negation.
\href{http://mishap.sdf.org/mints/mints-lifting-property-as-negation-DMG_5_no_4_2014.pdf}
		{The De Morgan Gazette \ensuremath{5} no.~4 (2014), 23--32,  ISSN 2053-1451}


\bibitem[Gavrilovich, Lifting Property]{Gavrilovich, Lifting property}
\newblock Misha Gavrilovich.
\newblock \href{http://mishap.sdf.org/mints/expressive-power-of-the-lifting-property.pdf}{Expressive power of the lifting property in elementary mathematics. A draft, current version.}
\newblock Arxiv arXiv:1707.06615 (7.17)


\bibitem[mintsGE]{mintsGE}
\newblock Misha Gavrilovich.
\newblock A naive diagram-chasing approach to formalisation of tame topology.
\newblock 2018.
\newblock \url{http://mishap.sdf.org/mintsGE.pdf}

\bibitem[situs]{situs}
\newblock An overview of the category of situses.
\newblock \url{http://ncatlab.org/nlab/show/situs}




\bibitem[Gromov, Ergobrain]{Gromov, Ergobrain}
\newblock Misha Gromov.
\newblock Structures, Learning and Ergosystems: Chapters 1-4, 6.
\newblock December 30, 2011.
\newblock \url{http://www.ihes.fr/~gromov/PDF/ergobrain.pdf}

\bibitem[Gleason]{Gleason}
\newblock Gleason, A.M.
\newblock Projective topological spaces.
\newblock  Illinois J. Math. 2(4A) pp. 482-489 (November 1958). DOI: 10.1215/ijm/1255454110  

\bibitem[PW]{PW}
\newblock Porter, J.R., Woods, R.G.  
\newblock Extremally Disconnected Spaces and Absolutes. Extensions and Absolutes of Hausdorff Space.
\newblock (1988).
\newblock 440–530. 
\newblock \href{doi:10.1007/978-1-4612-3712-9\_6}{https://doi.org/10.1007/978-1-4612-3712-9\_6}

\bibitem[Condensed]{Condensed}
\newblock Scholze, P.
\newblock \href{http://www.math.uni-bonn.de/people/scholze/Condensed.pdf}{Lectures on Condensed Mathematics.}
\newblock 2019.


\bibitem[Analytic]{Analytic}
\newblock Scholze, P.
\newblock \href{http://www.math.uni-bonn.de/people/scholze/Analytic.pdf}{Lectures on Analytic Geometry.}
\newblock 2019.


\bibitem[Strauss]{Strauss}
\newblock Strauss, D.P. 
\newblock Extremally Disconnected Spaces. 
\newblock Proceedings of the American Mathematical Society, 1967, 18(2), 305. 
\newblock \href{doi:10.2307/2035286}{https://doi.org/10.2307/2035286}


\bibitem[Taimanov]{Taimanov}
\newblock A. D. Taimanov. 
\newblock On extension of continuous mappings of topological spaces. 
\newblock Mat. Sb. (N.S.), 31(73):2 (1952), 459-463
\newblock \url{www.mathnet.ru/eng/sm5540} 


\end{thebibliography}
\end{document}